\documentclass[10pt]{article}
\usepackage{amsmath,amsfonts}
\normalbaselines
\newtheorem{teor}{Theorem}

\newtheorem{lema}[teor]{Lemma}
\newtheorem{coro}[teor]{Corollary}
\newtheorem{rem}[teor]{Remark}

\newcommand{\caja}{\null\hfill\rule{2mm}{2mm}}

\newcommand{\mb}{\mbox{$\overline{M}$}}

\title{Complete Spacelike Hypersurfaces in Generalized Robertson-Walker and the Null Convergence Condition. Calabi-Bernstein problems.}

\author{Juan A. Aledo${}^{a}$, Rafael M. Rubio${}^{b}$ and Juan J. Salamanca${}^{c}$ \\[6mm]
${}^a$ Departamento de Matem\'aticas, E.S.I. Inform\'atica, \\[0.5mm] Universidad de
Castilla-La Mancha, 02071 Albacete, Spain,\\ E-mail\textup{:
\texttt{juanangel.aledo@uclm.es}} \\[3mm]
${}^b, {}^c$ Departamento de Matem\'aticas, Campus de Rabanales, \\[0.5mm] Universidad de
C\'ordoba, 14071 C\'ordoba, Spain,\\[0.5mm] E-mails\textup{: \texttt{rmrubio@uco.es}, \texttt{jjsalamanca@uco.es}}\\[3mm]}

\date{}

\begin{document}

\maketitle

\thispagestyle{empty}

\begin{abstract}
We study constant mean curvature spacelike hypersurfaces in
generalized Robertson-Walker spacetimes $\overline{M}= I \times_f F$
which are spatially parabolic covered (i.e. its fiber $F$ is a
(non-compact) complete Riemannian manifold whose universal covering
is parabolic) and satisfy the null convergence condition. In
particular, we provide several rigidity results under appropriate
mathematical and physical assumptions. We pay special attention to
the case where the GRW spacetime is Einstein. As an application,
some Calabi-Bernstein type results are given.
\end{abstract}

\vspace*{5mm}

\noindent \textbf{MSC 2010:} 53C42, 53C50, 35J60.

\noindent {\it Keywords:} Spacelike hypersurface, constant mean
curvature, Generalized Robertson Walker spacetime, null convergence
condition, spatially parabolic spacetime, Calabi-Bernstein
problem.

\section{Introduction}

For a \emph{Generalized Robertson-Walker} (GRW) spacetime we mean a
product manifold $I\times F$ of an open interval $I$ of the real
line $\mathbb{R}$ endowed with the metric $dt^2$ and an $n(\geq
2)$-dimensional (connected) Riemannian manifold $(F, g_{_F})$,
furnished with the Lorentzian metric
\[
\overline{g} = -\pi^*_{_I} (dt^2) +f(\pi_{_I})^2 \, \pi_{_F}^*
(g_{_F}) \, ,
\]
where $\pi_{_I}$ and $\pi_{_F}$ denote the projections onto $I$ and
$F$, respectively, and $f$ is a positive smooth function on $I$
\cite{A-R-S1}. We will denote this  $(n+1)$-dimensional Lorentzian
manifold by $\overline{M}=I\times_f F$. So defined, $\overline{M}$
is a warped product in the sense of \cite[Chap. 7]{O'N}, with base
$(I,-dt^2)$, fiber $(F, g_{_F})$ and warping function $f$.  Observe
that the family of GRW spacetimes includes the classical
\emph{Robertson-Walker} (RW) spacetimes. Recall that in a RW
spacetime the fiber is 3-dimensional and of constant sectional
curvature, and the warping function (sometimes called
\emph{scale-factor}) can be thought, when the curvature sectional of
the fiber is positive, as the radius of the spatial universe
$\{t\}\times F$.

Note that a RW spacetime obeys the \emph{cosmological principle},
i.e. it is spatially homogeneous and spatially isotropic, at least
locally.  Thus, GRW spacetimes widely extend to RW spacetimes and
include, for instance, the Lorentz-Minkowski spacetime, the
Einstein-de Sitter spacetime, the Friedmann cosmological models, the
static Einstein spacetime and the de Sitter spacetime. GRW
spacetimes are useful to analyze if a property of a RW spacetime
$\overline{M}$ is \emph{stable}, i.e. if it remains true for
spacetimes close to $\overline{M}$ in a certain topology defined on
a suitable family of spacetimes \cite{Geroch}. Moreover, a conformal
change of the metric of a GRW spacetime with a conformal factor
which only depends on $t$, produces a new GRW spacetime.

Observe that a GRW spacetime is not necessarily spatially
homogeneous. Recall that spatial homogeneity seems appropriate just
as a rough approach to consider the universe in the large. However,
this assumption could not be realistic when the universe is
considered in a more accurate scale. Thus, these warped Lorentzian
manifolds become suitable spacetimes to model universes with
inhomogeneous spacelike geometries \cite{Ra-Sch}. A GRW spacetime
such that $f$ is constant will be called \emph{static}. Indeed, a
static GRW spacetime is in fact a Lorentzian product. On the other
hand, if the warping function $f$ is non-locally constant (i.e.
there is no open subinterval $J(\neq \emptyset)$ of $I$ such that
${f_{\mid}}_{J}$ is constant) then the GRW spacetime $\overline{M}$
is said to be \emph{proper}. This assumption means that there is no
(nonempty) open subset of $\overline{M}$ such that the sectional
curvature in $\overline{M}$ of any plane tangent to a
\emph{spacelike slice} $\{t\}\times F$ equals to the sectional
curvature of that plane in the inner geometry of the slice.

Any GRW spacetime has a smooth global time function and therefore it
is stably causal \cite[p. 64]{Beem}. If the fiber of a GRW spacetime
is compact, then it is called \emph{spatially closed}. Classically,
the subfamily of spatially closed GRW spacetimes has been very
useful to get closed cosmological models. On the other hand, a
number of observational and theoretical arguments on the total mass
balance of the universe \cite{Chiu} suggests the convenience of
adopting open cosmological models. Even more, a spatially closed GRW
spacetime violates the \emph{holographic principle} \cite[p.
839]{Bo} whereas a GRW spacetime with non-compact fiber could be a
suitable model compatible with that principle \cite{Bak-Rey}. There
again, nowadays is commonly accepted the theory of inflation. In
this setting, it is natural to think that expansion must occur in
the physical space at the same time and in the same manner. A
notable fact in this theory is that distant regions in our universe
cannot have any interaction. Notice that although the physical space
in instants after the inflation may not be exactly a model manifold,
in large scale the GRW spacetimes may be a good model to get an
approach to this reality.

In this work we are interested in the class of \emph{spatially
parabolic} GRW spacetimes. This notion was introduced and motivated
in \cite{RRS} as a natural counterpart of the spatially closed GRW
spacetimes. Spatially parabolic GRW spacetimes have a parabolic
Riemannian manifold as fiber, what provides a  significant wealth
from a geometric-analytic point of view. Recall that a complete
Riemannian manifold is parabolic if its only positive superharmonic
functions are the constants.

The importance in General Relativity of maximal and constant mean
curvature spacelike hypersurfaces in spacetimes is well-known; a
summary of several reasons justifying it can be found in \cite{M-T}.
In particular, hypersurfaces of (non-zero) constant mean curvature
are singularly suitable for studying the propagation of gravity
radiation \cite{S}. Classical papers dealing with uniqueness
problems for such kind of hypersurfaces are \cite{Ch}, \cite{BF} and
\cite{M-T}, although a previous relevant result in this direction
was the proof of the Calabi-Bernstein conjecture \cite{Calabi} for
maximal hypersurfaces in the $n$-dimensional Lorent-Minkowski
spacetime given by Cheng and Yau \cite{Cheng-Yau}. In \cite{BF},
Brill and Flaherty replaced the Lorent-Minkowski spacetime by a
spatial closed universe, and proved uniqueness in the large by
assuming ${\mathrm{Ric}}(z,z)>0$ for all timelike vectors $z$. In
\cite{M-T}, this energy condition was relaxed by Marsden and Tipler
to include, for instance, non-flat vacuum spacetimes. More recently,
Bartnik proved in \cite{Bar} very general existence theorems and
consequently, he claimed that it would be useful to find new
satisfactory uniqueness results. Still more recently, in
\cite{A-R-S1} Al\'\i as, Romero and S\'anchez gave new uniqueness
results in the class of spatially closed GRW spacetimes under the
Temporal Convergence Condition (TCC). In \cite{Ca-Ro-Ru2} several
known uniqueness results for compact CMC spacelike hypersurfaces in
GRW spacetimes were widely extended by means of new techniques to
the case of compact CMC spacelike hypersurfaces in spacetimes with a
timelike gradient conformal vector field. Finally, in \cite{RRS}
Romero, Rubio and Salamanca, obtained uniqueness results in the
maximal case for spatially parabolic GRW spacetimes under a
convexity property of the warping function.

\vspace{1mm}

Our main aim in this paper is to give new uniqueness results for
(non-compact) complete CMC hypersurfaces in spatially parabolic GRW
spacetimes which obey the Null Convergence Condition (NCC). As
known, the TCC is violated in inflationary spacetimes and so it is
natural to study uniqueness problems under the NCC, since some
inflationary scenarios can be modeled by spacetimes obeying this
energy condition. Moreover, certain class of GRW spacetimes obeying
the NCC arise as physically realistic cosmological models since they
satisfy the weak energy condition (see Section \ref{s5}). Some
recent papers dealing with uniqueness problems in GRW spacetimes
obeying the NCC under hypothesis relative to the curvatures of the
spacelike hypersurfaces are \cite{ARu}, \cite{AC}, \cite{ACL},
\cite{AlCL}, \cite{CCLP} and \cite{LP}

The paper is organized as follows. In Section \ref{s2} we revise
some notions regarding spacelike hypersurfaces in GRW spacetimes. In
Section \ref{s3} we provide several rigidity results for CMC
hypersurfaces in spatially parabolic covered GRW spacetimes (i.e.
its fiber $F$ is a (non-compact) complete Riemannian manifold whose
universal covering is parabolic) satisfying the NCC. We pay special
attention to the case when the GRW spacetime is Einstein, so
completing the characterization of compact CMC spacelike
hypersurfaces in spatially closed Einstein GRW spacetimes partially
developed in some previous papers (see \cite{A-R-S2} and
\cite{Ca-Ro-Ru2}), and extending this study to complete CMC
spacelike hypersurfaces in spatially parabolic covered Einstein GRW
spacetimes. Section \ref{s4} is devoted to provide several
Calabi-Bernstein results which follow from the former parametric
study. Finally, in Section \ref{s5} we justify the adequacy of GRW
spacetimes which satisfy the NCC condition to model some  physically
realistic cosmological universes.

\section{Preliminaries}\label{s2}

Let $(F,g_{_F})$ be an $n$-dimensional ($n\geq 2$) connected
Riemannian manifold and $I\subseteq \mathbb{R}$ an open interval in
$\mathbb{R}$ endowed with the metric $-dt^2$. The warped product
$\overline{M}= I \times_f F$ endowed with the Lorentzian metric
\begin{equation}\label{metrica}
\bar{g} = -\pi^*_{_I} (dt^2) +f(\pi_{_I})^2 \, \pi_{_F}^* (g_{_F})
\end{equation}
where $f>0$ is a smooth function on $I$, and $\pi_I$ and $\pi_F$
denote the projections onto $I$ and $F$ respectively, is said to be
a \emph{Generalized Robertson-Walker (GRW) spacetime} with
\emph{fiber} $(F,g_{_F})$, \emph{base} $(I,-dt^2)$ and \emph{warping
function} $f$ (see \cite{A-R-S1}).

The coordinate vector field $\partial_t:=\partial/\partial t$
globally defined on $\overline{M}$ is (unitary) timelike, and so
$\overline{M}$ is time-orientable. We will also consider on
$\overline{M}$ the conformal closed timelike vector field $K: =
f({\pi}_I)\,\partial_t$. From the relationship between the
Levi-Civita connections of $\overline{M}$ and those of the base and
the fiber \cite[Cor. 7.35]{O'N}, it follows that
\begin{equation}\label{conexion} \overline{\nabla}_XK =
f'({\pi}_I)\,X
\end{equation}
for any $X\in \mathfrak{X}(\overline{M})$, where $\overline{\nabla}$
is the Levi-Civita connection of the Lorentzian metric
(\ref{metrica}).

We will denote by $\overline{{\rm Ric}}$ the Ricci tensor of
$\overline{M}$. From \cite[Cor. 7.43]{O'N} it follows that
\begin{equation}\label{Ricb}
\overline{{\rm Ric}}(X,Y) ={\rm
Ric}^F(X^F,Y^F)+\left(\frac{f''}{f}+(n-1)\frac{f'^2}{f^2}\right)
\overline{g}(X^F,Y^F)-n\, \frac{f''}{f} \overline{g}(X,\partial_t)
\overline{g}(Y,\partial_t)
\end{equation}
for $X,Y\in \mathfrak{X}(\overline{M})$, where ${\rm Ric}^F$ stands
for the Ricci tensor of $F$. Here $X^F$ denotes the lift of the
projection of the vector field $X$ onto $F$, that is,
\[
X =X^F-\overline{g}(X,\partial_t)\partial_t.
\]

Regarding the scalar curvature $\overline{S}$ of $\overline{M}$, we
get from (\ref{Ricb}) that
\begin{equation}\label{Sb}
\overline{S}={\rm trace}(\overline{{\rm
Ric}})=\frac{S^F}{f^2}+2n\frac{f''}{f}+n(n-1)\frac{f'^2}{f^2},
\end{equation}
where $S^F$ stands for the scalar curvature of $F$.

Recall  that a Lorentzian manifold $\overline{M}$ obeys the
\emph{Null Convergence Condition (NCC)} if its Ricci tensor
$\overline{\rm Ric}$ satisfies $\overline{\rm Ric}(X,X) \geq 0$ for
all null vector $X\in \mathfrak{X}(\overline{M})$. In the case when
$\overline{M}= I \times_f F$ is a GRW spacetime, it can be checked
(see \cite{ARu}) that $\overline{M}$ obeys the NCC if and only if
\begin{equation}\label{NCCc}
Ric^F-(n-1)f^2 (\log f)''\geq 0,
\end{equation}
where $Ric^F$ stands for the Ricci curvature of $(F,g_{_F})$. Recall
that the Ricci curvature at each point $p\in F$ in the direction
$X(p)\in T_pF$, $X\in\mathfrak{X}(F)$, is defined as
\[
Ric^F(X(p))=\frac{{\rm Ric}^F(X(p),X(p))}{g_{_F}(X(p),X(p))}={\rm
Ric}^F\Big(\frac{X(p)}{\mid X(p)\mid_{_F}},\frac{X(p)}{\mid
X(p)\mid_{_F}}\Big).
\]

On the other hand, we will say that a spacetime $\overline{M}$
verifies the \emph{NCC with strict inequality} if its Ricci tensor
$\overline{\rm Ric}$ satisfies $\overline{\rm Ric}(X,X) > 0$ for all
null vector $X\in \mathfrak{X}(\overline{M})$. Now, a GRW spacetime
$\overline{M}= I \times_f F$ obeys the NCC with strict inequality if
and only if $Ric^F-(n-1) f^2 (\log f)''> 0$.

A smooth immersion $\psi:M^n\longrightarrow\mb$ of an
$n$-dimensional (connected) manifold $M$ is said to be a
\emph{spacelike hypersurface} if the induced metric via $\psi$ is a
Riemannian  metric $g$ on $M$.

Since $\overline{M}$ is time-orientable we can take, for each
spacelike hypersurface $M$ in $\overline{M}$, a unique unitary
timelike vector field $N \in \mathfrak{X}^\bot(M)$ globally defined
on $M$ with the same time-orientation as $\partial_t$, i.e. such
that $\bar{g}(N,\partial_t)<0$. From the wrong-way Cauchy-Schwarz
inequality (see \cite[Prop. 5.30]{O'N}, for instance), we have
$\bar{g}( N,
\partial_t) \leq -1$, and the equality holds at a point $p\in M$ if
and only if $N = \partial_t$ at $p$. The \emph{hyperbolic angle}
$\varphi$, at any point of $M$, between the unit timelike vectors
$N$ and $\partial_t$, is given by $\bar{g}(N,\partial_t)=-\cosh
\varphi$. This angle has a reasonable physical interpretation. In
fact, in a GRW spacetime $\overline{M}$ the integral curves of
$\partial_t$ are called \emph{comoving observers} \cite[p. 43]{SW2}.
If $p$ is a point of a spacelike hypersurface $M$ in $\overline{M}$,
among the instantaneous observers at $p$, $\partial_t(p)$ and
$N_{_p}$ appear naturally. In this sense, observe that the energy
$e(p)$ and the speed $v(p)$ that $\partial_t(p)$ measures for
$N_{_p}$ are given, respectively, by $e(p)=\cosh\varphi(p)$ and
$|v(p)|^2=\tanh^2\varphi(p)$ \cite[pp. 45-67]{SW2}.

%
%
%

We will denote by $A$ and $H:= -(1/n) \mathrm{tr}(A)$ the
\emph{shape operator} and the \emph{mean curvature function}
associated to $N$. A spacelike hypersurface with $H=0$ is called a
\emph{maximal} hypersurface. The reason for this terminology is that
the mean curvature is zero if and only if the spacelike hypersurface
is a local maximum of the $n$-dimensional area functional for
compactly supported normal variations.

In any GRW spacetime $\overline{M}$ there is a remarkable family of
spacelike hypersurfaces, namely its spacelike \emph{slices}
$\{t_{_0}\}\times F$, $t_{_0}\in I$. The spacelike slices constitute
for each value $t_{_0}$ the restspace of the distinguished observers
in $\partial_t$. A spacelike hypersurface in $\overline{M}$ is a
(piece of) spacelike slice if and only if the function $\tau:=\pi_I
\circ \psi$ is constant. Furthermore, a spacelike hypersurface in
$\overline{M}$ is a (piece of) spacelike slice if and only if the
hyperbolic angle $\varphi$ vanishes identically. The shape operator
of the spacelike slice $\tau=t_{_0}$ is given by
$A=-f'(t_{_0})/f(t_{_0})\,I$, where $I$ denotes the identity
transformation, and so its (constant) mean curvature is $H=
 f'(t_{_0})/f(t_{_0})$. Thus, a spacelike slice is maximal if and
only if $f'(t_{_0})=0$ (and hence, totally geodesic). We will say
that the spacelike hypersurface is contained in a \emph{slab}, if it
is contained between two spacelike slices.

If we put $\partial_t^T=\partial_t+\overline{g}(\partial_t,N)N$ the
tangential part of $\partial_t$ and $N^F
=N+\overline{g}(N,\partial_t)\partial_t$, it follows from
$\overline{g}(N,N)=-1=\overline{g}(\partial_t,\partial_t)$ that
\begin{equation}\label{paraE}
\left|\partial_t^T\right|^2=\left|N^F\right|^2= \sinh^2\varphi.
\end{equation}
Hence, a spacelike hypersurface in $\overline{M}$ is a (piece of)
spacelike slice if and only if
$\left|\partial_t^T\right|^2=\left|N^F\right|^2$ vanishes
identically on $M$.

To finish this section, let us briefly revise some important notions
on parabolicity in GRW spacetimes. Recall that a GRW spacetime
$\overline{M}=I\times_f F$ is said to be \emph{spatially parabolic}
\cite{RRS} if its fiber is \emph{parabolic}; i.e. it is a
non-compact complete Riemannian manifold such that the only
superharmonic functions on it which are bounded from below are the
constants. Analogously, a GRW spacetime is said to be
\emph{spatially parabolic covered} if its universal Lorentzian
covering is spatially parabolic. Observe that the universal
Lorentzian covering of $I\times_f F$ is $I\times_f \widetilde{F}$,
where $\widetilde{F}$ is the universal Riemannian covering of the
fiber $F$. In particular, every spatially parabolic covered GRW
spacetime is spatially parabolic, and both notions agree on a GRW
spacetime with a simply-connected fiber. GRW spacetimes which admit
a complete parabolic spacelike hypersurface have been studied in
\cite{RRS}, where the following result is proved:

\begin{quote} {\it Let $M$ be a complete spacelike hypersurface
in a spatially parabolic covered GRW spacetime $\overline{M}= I
\times_f F$. If the hyperbolic angle of $M$ is bounded and the
restriction $f(\tau)$ on $M$ of the warping function $f$ satisfies:
\begin{itemize}
\item[i)] $\sup f(\tau)<\infty$, and
\item[ii)] $\inf f(\tau)>0,$
\end{itemize}
then, $M$ is parabolic. }
\end{quote}
This result will be used in Section \ref{s3}.

\section{Parametric type results}\label{s3}
Let $\psi: M \rightarrow \overline{M}$ be a spacelike hypersurface
in a GRW spacetime $\overline{M}= I \times_f F$. It is easy to check
that the gradient of $\tau=\pi_I \circ \psi$ on $M$ is given by
\begin{equation}\label{part}
\nabla \tau=-\partial_t^T
\end{equation}
and its Laplacian by
\begin{equation}\label{laptau}
\Delta \tau = - \frac{f'(\tau)}{f(\tau)} \left\{ n + |\nabla \tau|^2
\right\} - n H \, \overline{g}(N, \partial_t).
\end{equation}

Let us take $G:I\longrightarrow\mathbb{R}$ such that $G'=f$. Using
(\ref{part}) we have that the gradient of $G(\tau)$ on $M$ is given
by
\begin{equation}\label{otra}
\nabla G(\tau)=G'(\tau)\nabla \tau=-f(\tau) \partial_t^T=-K^T,
\end{equation}
where $K^T=K+\overline{g}(K,N)N$ is the tangential component of $K$
along $\psi$, and so its Laplacian on $M$ (see \cite[Eq.
6]{A-R-S1})) yields
\begin{equation}\label{LG}
\Delta G(\tau)={\rm div}(\nabla
G(\tau))=-nf'(\tau)-nH\overline{g}(K,N).
\end{equation}

As a consequence of (\ref{LG}) we have
\begin{teor} Let $\overline{M}= I \times_f F$ be a spatially
parabolic covered GRW spacetime and $\psi: M \rightarrow
\overline{M}$ a complete spacelike hypersurface which is contained
in a slab and whose hyperbolic angle is bounded. If the mean
curvature of $M$ satisfies that $Hf'(\tau)\leq 0$, then $M$ is a
maximal slice.
\end{teor}
\begin{proof}
Since $Hf'(\tau)\leq 0$ it follows that the bounded function
$G(\tau)$ has signed Laplacian, and therefore $G(\tau)$ is constant.
Then, from (\ref{otra}) and (\ref{paraE}) we conclude that $M$ is a
spacelike slice. Finally, since the mean curvature of a slice
$\{t_{_0}\}\times F$ is $H=
 f'(t_{_0})/f(t_{_0})$, it must be $H=0$, i.e. $M$ is a maximal
 slice. \hfill{$\Box$}
\end{proof}

Another immediate consequence of (\ref{LG}) is the following result

\begin{teor} Let $\overline{M}= I \times_f F$ be a spatially
parabolic covered GRW spacetime and $\psi: M \rightarrow
\overline{M}$ a complete spacelike hypersurface which is contained
in a slab and whose hyperbolic angle is bounded. If the mean
curvature of $M$ satisfies that $H\geq \frac{f'(\tau)^2}{f(\tau)^2}$
and either $H\geq 0$ or $\frac{f'}{f\cosh\varphi}\leq H\leq 0$, then
$M$ is a spacelike slice.
\end{teor}
\begin{proof}
It is easu to check that under the assumptions on $H$ the Lapacian
of $G(\tau)$ has sign, and therefore $G(\tau)$ is constant. Again,
from (\ref{otra}) and (\ref{paraE}) we conclude that $M$ is a
spacelike slice.  \hfill{$\Box$}
\end{proof}

\begin{rem}{\rm The inequality
$H^2\geq \frac{f'(\tau)^2}{f(\tau)^2}$ can be geometrically
interpreted as follows: the mean curvature of the spacelike
hypersurface, at any point is,  in absolute value, greater or equal
than the mean curvature of the spacelike slice at that point.}
\end{rem}

A direct computation from (\ref{conexion}) gives
\[
\nabla \overline{g}(K,N)=-AK^T,
\]
where we have also used (\ref{part}), and so the Laplacian of
$\overline{g}(K,N)$ on $M$ becomes (see \cite[Eq. 8]{A-R-S1})
\begin{equation}\label{LKN}
\Delta \overline{g}(K,N)={\rm div}(\nabla \overline{g}(K,N))=
\overline{{\rm Ric}}(K^T,N)+n\overline{g}(\nabla H,K)+nf'(\tau)H+
\overline{g}(K,N){\rm tr}(A^2).
\end{equation}

On the other hand, from
(\ref{Ricb}) we have
\begin{eqnarray}
\overline{{\rm Ric}}(K^T,N) & = & \overline{g}(K,N) \,
\overline{{\rm Ric}}(N^F,N^F)-\overline{g}(K,N)
\left|\partial_t^T\right|^2\,
\overline{{\rm Ric}}(\partial_t,\partial_t) \nonumber \\
& = & \overline{g}(K,N)\left( {\rm Ric}^F(N^F,N^F)-(n-1)
\left|N^F\right|^2\,(\log f)''(\tau)\right) \nonumber \\
& = & \overline{g}(K,N) \left|N^F\right|^2_{_F} \left(
Ric^F\left({N^F}\right)-(n-1) f^2(\tau)\,(\log f)''(\tau)\right),
\label{RicKN}
\end{eqnarray}
where $\left|N^F\right|_{_F}=g_{_F}(N^F,N^F)^{1/2}$. In particular,
observe that if $\overline{M}$ obeys the NCC then $\overline{{\rm
Ric}}(K^T,N)\leq 0$. Furthermore, if $\overline{M}$ obeys the NCC
with strict inequality, then $\overline{{\rm Ric}}(K^T,N)\equiv 0$
if and only if $M$ is a (piece of) spacelike slice (see
(\ref{paraE})).

Then, from (\ref{LG}), (\ref{LKN}) and (\ref{RicKN}), we get

\begin{lema}\label{l1} Let $\psi: M \rightarrow \overline{M}$ be a constant mean curvature spacelike hypersurface
in a GRW spacetime $\overline{M}= I \times_f F$, and
$G:I\longrightarrow\mathbb{R}$ such that $G'=f$. Then
\begin{eqnarray*}
\Delta (HG(\tau)+\overline{g}(K,N))& = & -\overline{g}(K,N) \left\{ nH^2-{\rm tr}(A^2) \right.  \\
&&\left. -|N^F|^2_{_F}\left(Ric^F\left({N^F}\right)-(n-1)
f^2(\tau)\,(\log f)''(\tau)\right)\right\}.
\end{eqnarray*}
In particular, if $\overline{M}$ obeys the NCC then $\Delta
(HG(\tau)+\overline{g}(K,N))\leq 0$.
\end{lema}

\vspace{2mm}

Let $\overline{M}= I \times_f F$ be a spatially parabolic covered
GRW spacetime obeying the NCC. From the study developed above, next
we will provide several rigidity results for CMC complete spacelike
hypersurfaces in $\overline{M}$. In some of these results, in order
to derive the parabolicity of the spacelike hypersurface it is used
that the assumptions $\inf f(\tau)>0$ and $\sup f(\tau)<\infty$ are
automatically satisfied if the hypersurface is contained in a slab.

\begin{teor} \label{t1} Let $\overline{M}= I \times_f F$ be a spatially
parabolic covered GRW spacetime obeying the NCC and $\psi: M
\rightarrow \overline{M}$ a complete CMC spacelike hypersurface
which is contained in a slab and whose hyperbolic angle is bounded.
Then $M$ is totally umbilical.
\end{teor}
\begin{proof}
Observe that, since $M$ is contained between two spacelike slices,
both $G(\tau)$ and  $f(\tau)$ are bounded, being also $\inf
f(\tau)>0$. As said in Section \ref{s2}, under the assumptions above
it follows that $M$ is parabolic. Then, since
$HG(\tau)+\overline{g}(K,N)$ is a bounded function on $M$ whose
Laplacian is non positive (see Lemma \ref{l1}), we conclude that
such Laplacian must vanish identically and consequently $nH^2-{\rm
tr}(A^2)\equiv 0$ on $M$, i.e. $M$ is totally umbilical.
\hfill{$\Box$}
\end{proof}


On the other hand, we can conclude that the spacelike hypersurface
is a spacelike slice by asking the spacetime to obey the NCC with
strict inequality.

\begin{teor} \label{t2} Let $\overline{M}= I \times_f F$ be a spatially
parabolic covered GRW spacetime obeying the NCC with strict
inequality and $\psi: M \rightarrow \overline{M}$ a complete CMC
spacelike hypersurface which is contained in a slab and whose
hyperbolic angle is bounded. Then $M$ is a spacelike slice.
\end{teor}
\begin{proof}
Note that, under this additional assumption, it must be
$\left|N^F\right|^2\equiv 0$ on $M$, which implies (see
(\ref{paraE})) that $M$ is a spacelike slice. \hfill{$\Box$}
\end{proof}

For the particular case when $M$  is maximal, we have

\begin{coro} \label{c1} Let $\overline{M}= I \times_f F$ be a spatially
parabolic covered GRW spacetime obeying the NCC and $\psi: M
\rightarrow \overline{M}$ a complete maximal spacelike hypersurface
which is contained in a slab and whose hyperbolic angle is bounded.
Then $M$ is totally geodesic.
\end{coro}

\begin{coro} \label{c2} Let $\overline{M}= I \times_f F$ be a spatially
parabolic covered GRW spacetime obeying the NCC with strict
inequality and $\psi: M \rightarrow \overline{M}$ a complete maximal
spacelike hypersurface which is contained in a slab and whose
hyperbolic angle is bounded. Then $M$ is a totally geodesic
spacelike slice.
\end{coro}

Recall that a GRW spacetime is said to be \emph{proper} if the
warping function $f$ is non-locally constant, i.e. there is no open
subinterval $J(\neq \emptyset)$ of $I$ such that ${f_{\mid}}_{J}$ is
constant. Next we characterize the spacelike slices of a proper
spatially parabolic covered GRW spacetime obeying the NCC by means
of a pinching condition for its (constant) mean curvature $H$.

\begin{teor} \label{t3}
Let $\overline{M}= I \times_f F$ be a proper spatially parabolic
covered GRW spacetime obeying the NCC and $\psi: M \rightarrow
\overline{M}$ a complete CMC spacelike hypersurface whose hyperbolic
angle is bounded. If the mean curvature function of $M$ satisfies
that $H^2\geq \frac{f'(\tau)^2}{f(\tau)^2}$ and the restriction
$f(\tau)$ of the warping function $f$ on $M$ is such that $\inf
f(\tau)>0$ and $\sup f(\tau)<\infty$, then $M$ is a spacelike slice
$(\tau=t_{_0})$ with $H^2=\frac{f'(t_{_0})^2}{f(t_{_0})^2}$.
\end{teor}
\begin{proof}
Since the hyperbolic angle of $M$ is bounded and $f(\tau)$ satisfies
that $\inf f(\tau)>0$ and $\sup f(\tau)<\infty$, we conclude that
$M$ is parabolic (see Section \ref{s2}).

From the assumption on the mean curvature of $M$ we have that
\[
\mid H\mid\geq\frac{\mid f'(\tau)\mid}{f(\tau)},
\]
and so
\[
{\rm tr}(A^2)\geq nH^2\geq \frac{n}{f(\tau)}\mid f'(\tau)H\mid.
\]
Then
\[
nf'(\tau)H+ \overline{g}(K,N){\rm tr}(A^2)\leq 0,
\]
which implies that the Laplacian of $\overline{g}(K,N)$ (\ref{LKN})
is non positive and consequently constant.

Moreover
\[
\mid nf'(\tau)H\mid=\mid \overline{g}(N,K) \mid {\rm tr}(A^2)\geq
f\,{\rm tr}(A^2)\geq \mid nf'(\tau)H\mid,
\]
and therefore $f=\mid \overline{g}(N,K) \mid=f(\tau)\cosh \varphi$.
Consequently $\varphi$ vanishes identically on $M$, which means that
$M$ is a spacelike slice. \hfill{$\Box$}
\end{proof}

As commented in the introduction, a GRW spacetime is spatially
closed if its fiber $F$ is compact \cite[Prop. 3.2]{A-R-S1}. Since
on a compact Riemannian manifold the only functions with signed
Laplacian are the constants, reasoning as in Theorem \ref{t3} it can
be proved the following

\begin{teor} \label{paradesitter}
Let $\overline{M}= I \times_f F$ be a proper spatially closed GRW
spacetime obeying the NCC and $\psi: M \rightarrow \overline{M}$ a
compact CMC spacelike hypersurface whose mean curvature satisfies
that $H^2\geq \frac{f'(\tau)^2}{f(\tau)^2}$. Then $M$ is a spacelike
slice $(\tau=t_{_0})$ with $H^2=\frac{f'(t_{_0})^2}{f(t_{_0})^2}$.
\end{teor}

A relevant example of proper spatially closed GRW spacetime obeying
the NCC is the de Sitter spacetime which, in its intrinsic version
is given as the Robertson-Walker spacetime
$\mathbb{S}^{n+1}_1=\mathbb{R}\times_{\cosh t}\mathbb{S}^ n$. In
\cite[Theorem 1]{AA1} the authors established a sufficient condition
for a compact spacelike in $\mathbb{S}^{n+1}_1$ (considered as an
hyperquadric of the $(n+2)$-dimensional Lorent-Minkowski spacetime)
to be totally umbilical, in terms of a lower bound for the squared
of its mean curvature. As a consequence of Theorem
\ref{paradesitter}, we obtain the following intrinsic approach of
the previously cited result:

\begin{coro}
Let $\psi: M \rightarrow \mathbb{S}^{n+1}_1$ be a spacelike
hypersurface in the de Sitter spacetime whose constant mean
curvature satisfies that $H^2\geq \tanh^2(\tau)$. Then $M$ is a
spacelike slice with $H^2=\tanh^2(\tau)$.
\end{coro}

Notice that in $\mathbb{S}^{n+1}_1$ there exists an only maximal
slice and, for any $t\neq 0$, exactly two spacelike slices with
$H^2=\tanh^2(t)$.

\vspace{2mm}

Next, we provide another uniqueness result under the hypothesis of
monotony of the warping function.

\begin{teor}\label{tmon}
Let $\overline{M}= I \times_f F$ be a spatially parabolic covered
GRW spacetime obeying the NCC,  and let $\psi: M \rightarrow
\overline{M}$ be a complete CMC spacelike hypersurface whose
hyperbolic angle is bounded and such that $\sup f(\tau) <\infty$ and
$\inf f(\tau) >0$.

If the restriction of $f$ to $\tau(M)$ is non-increasing (resp. non
decreasing) and $H\geq 0$ (resp. $H\leq 0$), then $M$ is totally
geodesic.
\end{teor}
\begin{proof}
From (\ref{LKN}) we have that $\overline{g}(K,N)$ is subharmonic on
the parabolic manifold $(M,g)$. Since moreover that function is
bounded, it must be constant. Finally, using again (\ref{LKN}) it
follows that ${\rm tr}(A^2)$ vanishes identically and therefore $M$
is totally geodesic. \hfill{$\Box$}
\end{proof}

In the above theorem, if we ask $\overline{M}= I \times_f F$ to obey
the NCC with strict inequality, then we conclude that $M$ is a
totally geodesic spacelike slice.

\vspace{2mm}


Next we provide another rigidity result (Theorem \ref{new} for
complete CMC spacelike hypersurfaces in GRW spacetimes whose fiber
has its sectional curvature bounded from below and whose warping
function $f$ satisfies that $(\log f)''\leq 0$. Note that the NCC
will be not required in this theorem. In order to do that, we will
need the following result which extends \cite[Lemma 13]{ARR}. In
fact, note that in such Lemma the fiber is asked to have
non-negative sectional curvature, whereas in the following result
this assumption changes to have sectional curvature bounded from
below.

\begin{lema} \label{mejor13} Let $\psi: M \rightarrow \overline{M}$ be a
complete CMC spacelike hypersurface in a GRW spacetime
$\overline{M}= I \times_f F$ whose warping function satisfies $(\log
f)''\leq 0$ and whose fiber has its sectional curvature bounded from
below. Then the Ricci curvature of $M$ is bounded from below.
\end{lema}
\begin{proof}
Given $Y\in \mathfrak{X}(M)$ such that $g(Y,Y)=1$, let us write
\[
Y=-\overline{g}(\partial_t,Y)\partial_t+Y^F.
\]
From the Schwarz inequality, we get using (\ref{part}) and
(\ref{paraE}) that
\[
\overline{g}(\partial_t,Y)^2=g(\nabla\tau,Y)^2\leq\mid\nabla\tau\mid^2=\sinh^2\varphi.
\]
As a consequence, $\mid Y^F\mid^2=1+\overline{g}(\partial_t,Y)^2$ is
bounded.

Given $p\in M$, let us take a local orthonormal frame
$\{U_1,...,U_n\}$ around $p$. From the Gauss equation
\[
\langle R(X,Z)V,W\rangle=\langle \overline{R}(X,Z)V,W\rangle+\langle
AZ,W\rangle \langle AX,V \rangle - \langle AZ,V\rangle \langle
AX,W\rangle, \quad X,Z,V,W\in \mathfrak{X}(M)
\]
where $\overline{R}$ and $R$ denote the curvature tensors of
$\overline{M}$ and $M$ respectively, and $A$ is the shape operator
of $\psi$, we get that the Ricci curvature of $M$, ${\rm Ric}^M$,
satisfies
\[
{\rm Ric}^M(Y,Y)\geq \sum_k
\overline{g}(\overline{R}(Y,U_k)Y,U_k)-\frac{n^2}{4}H^2\vert
Y\vert^2, \quad Y\in \mathfrak{X}(M), \ \ g(Y,Y)=1.
\]
Now, from \cite[Proposition 7.42]{O'N} we have
\begin{eqnarray*}
\sum_{k=1}^n \overline{g}(\overline{R}(Y,U_k)Y,U_k)&
=&\sum_{k=1}^ng_{_F}(R^F(Y^F, U_k^F)Y^F,U_k^F)+ (n-1)\frac{f'^2}{f^2} \\
&& -(n-2)(\log f)''g(Y,\nabla\tau)^2 -(\log
f)''\vert\nabla\tau\vert^2,
\end{eqnarray*}
where $R^F$ denotes the curvature tensor of the fiber $F$. Since the
sectional curvature of $F$ is bounded from below, there exists a
constant $C$ such that $\sum_{k=1}^n
\overline{g}(\overline{R}(Y,U_k)Y,U_k)\geq C$. Therefore
\[
{\rm Ric}^M(Y,Y)\geq C -\frac{n^2}{4}H^2,
\]
namely, the Ricci curvature of $M$ is bounded from below as we
wanted to prove. \hfill{$\Box$}
\end{proof}

To demonstrate Theorem \ref{new} we will use \cite[Lemma 12]{ARR}.
To facilitate the understanding of its proof, observe that in the
paper \cite{ARR} the hypersurface $\psi: M \rightarrow \overline{M}$
was oriented by choosing the Gauss map $N$ such that
$\bar{g}(N,\partial_t)>0$. This change of orientation means that,
according to the orientation chosen in the present article, the
thesis of \cite[Lemma 12]{ARR} becomes $H=f'(\tau)/f(\tau)$.

\begin{teor}\label{new} Let $\overline{M}= I \times_f F$ be a spatially
parabolic covered GRW spacetime whose warping function satisfies
$(\log f)''\leq 0$ and whose fiber has its sectional curvature
bounded from below. Let $\psi: M \rightarrow \overline{M}$ be a
complete CMC spacelike hypersurface which is contained in a slab and
whose hyperbolic angle is bounded. Then $M$ is a spacelike slice.
\end{teor}
\begin{proof}
From the assumptions it follows using Lemma \ref{mejor13} and
\cite[Lemma 12]{ARR} that
\[
H=\frac{f'(\tau)}{f(\tau)}.
\]
Now, using (\ref{LG}) we obtain
\[
\Delta G(\tau)=-nf(\tau)(-H+H\cosh\varphi)\leq 0.
\]

Taking into account the boundedness of the function $G(\tau)$ and
the parabolicity of $M$, we have that $G(\tau)$ must be constant and
$\nabla G(\tau)=-f(\tau){\partial_t}^T=0$, namely $M$ is a spacelike
slice. \hfill{$\Box$}
%
%
%
%
%
%
\end{proof}

\begin{rem} {\rm Observe that Theorem \ref{new} widely improves
\cite[Theorem 14]{ARR} in many aspects:
\begin{itemize}
\item In \cite[Theorem 14]{ARR} the dimension of $M$ is restricted to $n\leq
4$, whereas in Theorem \ref{new} this dimension is arbitrary.
\item In \cite[Theorem 14]{ARR} the fiber is asked to have
non-negative sectional curvature, whereas in Theorem \ref{new} this
assumption changes to have sectional curvature bounded from below.
\item In \cite[Theorem 14]{ARR} the warped function $f$ is asked to satisfy $f''(\tau)\leq 0$,
whereas in Theorem \ref{new} this assumption changes to the weaker
one $(\log f)''(\tau)\leq 0$.
\item Finally, in contrast to \cite[Theorem 14]{ARR}, in Theorem \ref{new} the maximal case is included.
\end{itemize}}
\end{rem}

In \cite[Section 4]{AA}, Albujer and Al\'\i as introduced the notion
of \emph{steady state type spacetimes}, as the warped products with
fiber an $n$-dimensional Riemannian manifold $(F,g_{_F})$, base
$(\mathbb{R},-dt^2)$ and warping function $f(t)=e^t$. This family
contains, for instance, the De Sitter cusp \cite{Ga}. In particular,
these GRW spacetimes obey the NCC provided that the fiber $F$ has
non-negative Ricci curvature. As a consequence of our Theorem \ref{new}, we can enunciate

\begin{quote}{\it Let $\overline{M}=\mathbb{R}\times_{e^ t} F$ be a spatially parabolic
stedy state type spacetime, whose fiber has non-negative Ricci curvature.  Let $\psi: M \rightarrow \overline{M}$ be a
complete CMC spacelike hypersurface which is contained in a slab and
whose hyperbolic angle is bounded. Then $M$ is a spacelike slice.}
\end{quote}

This result extends \cite[Th. 8]{AA} to arbitrary dimension. In
fact, in \cite[Th. 8]{AA} the authors obtain the same rigidity
result when the fiber has dimension 2 using that a complete
2-dimensional Riemannian manifold whose Gaussian curvature is
non-negative is parabolic.

\subsection{Einstein GRW spacetimes}
Recall that a spacetime $(\overline{M},\overline{g})$  is called
\emph{Einstein} if its Ricci tensor $\overline{{\rm Ric}}$ is
proportional to the metric $\overline{g}$. When
$\overline{M}=I\times_f F$ is a GRW spacetime, it is well-known that
$\overline{M}$ is Einstein with $\overline{{\rm Ric}}=\overline{c}\,
\overline{g}$, $\overline{c} \in \mathbb{R}$, if and only if the
fiber $(F,g_{_F})$ has constant Ricci curvature $c$ and the warping
function $f$ satisfies the differential equations
\begin{equation}\label{Ein}
\frac{f''}{f}=\frac{\overline{c}}{n} \hspace{1cm} {\rm and}
\hspace{1cm} \frac{\overline{c} (n-1)}{n}=\frac{c+(n-1)(f')^2}{f^2},
\end{equation}
which, in particular, imply that $(n-1)(\log f)''=\frac{c}{f^ 2}$
(see \cite[section 6]{Ca-Ro-Ru2}). Obviously, every Einstein
spacetime obeys the NCC.

All the positive solutions to (\ref{Ein}) were collected in
\cite{A-R-S2}. For the sake of completeness, we show such
classification in Table \ref{table1}

\begin{table}
\caption{Warping functions for Einstein GRW
spacetimes}\label{table1}
\begin{tabular}{llll}
\\
\hline \\
1 &  $\overline{c}>0$ & $c>0$ & $f(t)=a\, e^{bt}+\frac{cn}{4a\overline{c}(n-1)}\,e^{-bt}$,\quad $a>0$,\quad $b=\sqrt{\overline{c}/n}$ \\
2 &  $\overline{c}>0$ & $c=0$ & $f(t)=a\,e^{\varepsilon b t}$, \quad $a>0$,\quad $\varepsilon=\pm 1$, \quad $b=\sqrt{\overline{c}/n}$ \\
3 &  $\overline{c}>0$ & $c<0$ & $f(t)=a\, e^{bt}+\frac{cn}{4a\overline{c}(n-1)}\,e^{-bt}$,\quad $a\neq 0$,\quad $b=\sqrt{\overline{c}/n}$ \\
4 &  $\overline{c}=0$ & $c=0$ & $f(t)=a$,\quad $a>0$ \\
5 &  $\overline{c}=0$ & $c<0$ & $f(t)=\varepsilon \sqrt{\frac{-c}{n-1}}\, t+a$, \quad $\varepsilon=\pm 1$ \\
6 &  $\overline{c}<0$ & $c<0$ & $f(t)=a_1\cos(bt)+a_2\sin(bt)$, \quad $a_1^2+a_2^2=\frac{cn}{\overline{c}(n-1)}$, \quad $b=\sqrt{-\overline{c}/n}$\\ \\
\hline
\end{tabular}
\end{table}


In \cite[Theorem 6.1]{Ca-Ro-Ru2}, the authors proved that the
spacelike slices are the only compact CMC spacelike hypersurfaces in
an Einstein GRW spacetime whose fiber has Ricci curvature $c\leq 0$.
This result covers the cases 2-6 in Table \ref{table1}. However, the
techniques used there cannot be applied to study the first case
($\overline{c}>0$ and $c>0$). For these values, from the
Bonnet-Myers Theorem we have that the fiber $F$ is compact, and so
the GRW spacetime is spatially closed.

Since on a compact Riemannian manifold the only functions with
signed Laplacian are the constants, as a direct consequence of the
proof of Theorem \ref{t1} we conclude that
\begin{quote}\emph{
Every compact CMC spacelike hypersurface in an Einstein GRW
spacetime whose fiber has positive Ricci curvature $c> 0$ is totally
umbilical.}
\end{quote}

Actually, this is the best possible result. In fact, recall that the
de Sitter spacetime has a realization as the GRW spacetime
$\mathbb{S}^{n+1}_1=\mathbb{R}\times_{\cosh t}\mathbb{S}^ n$. In
particular, $\mathbb{S}^{n+1}_1$ is included in the case 1 of Table
\ref{table1} and, as is well-known, it contains compact CMC
spacelike hypersurfaces which are not spacelike slices.

Also observe that Theorem \ref{t1} allows to extend the previous
study from the compact case to the one of complete CMC spacelike
hypersurfaces in a spatially parabolic covered Einstein GRW
spacetime, being able to consider jointly the six cases mentioned
above. Specifically, we have the following corollary which widely
extend \cite[Theorem 6.1]{Ca-Ro-Ru2} and the rigidity results in
\cite{A-R-S2}
\begin{coro}\label{c6t} Let $\overline{M}= I \times_f F$ be a spatially
parabolic covered Einstein GRW spacetime and $\psi: M \rightarrow
\overline{M}$ a complete CMC spacelike hypersurface which is
contained in a slab and whose hyperbolic angle is bounded. Then $M$
is totally umbilical.
\end{coro}

Anyway, we are able to go further in the cases 2-6. In fact, note
that in these cases the warping function $f$ satisfies that $(\log
f)''\leq 0$. Then, if additionally we ask the fiber $F$ to have its
sectional curvature bounded from below we have
\begin{coro}\label{corE} Let $\overline{M}= I \times_f F$ be a spatially
parabolic covered Einstein GRW spacetime whose fiber has Ricci
curvature $c\leq 0$ (cases 2-6 in Table \ref{table1}) and whose
sectional curvature is bounded from below. Let $\psi: M \rightarrow
\overline{M}$ be a complete CMC spacelike hypersurface which is
contained in a slab and whose hyperbolic angle is bounded. Then $M$
is a spacelike slice.
\end{coro}

\section{Calabi-Bernstein type Problems} \label{s4}
Let $(F,g_{_F})$ be a (non-compact) $n$-dimensional Riemannian
manifold and $f : I \longrightarrow \mathbb{R}$ a positive smooth
function. For each $u \in C^{\infty}(F)$ such that $u(F)\subseteq
I$, we can consider its graph $\Sigma_u=\{(u(p),p) \, : \, p\in F\}$
in the Lorentzian warped product $(\overline{M}=I\times_f
F,\overline{g})$. The graph inherits from $\overline{M}$ a metric,
represented on $F$ by
\[
g_u=-du^2+f(u)^2g_{_F}.
\]
This metric is Riemannian (i.e. positive definite) if and only if
$u$ satisfies $|Du|<f(u)$ everywhere on $F$, where $Du$ denotes the
gradient of $u$ in $(F,g_{_F})$ and $| Du|^2=g_{_F}(Du,Du)$. Note
that $\tau(u(p),p)=u(p)$ for any  $p \in F$, and so $\tau$ and $u$
may be naturally identified on $\Sigma_u$.

When $\Sigma_u$ is spacelike, the unitary normal vector field on
$\Sigma_u$ satisfying $\overline{g}( N,\partial_t)<0$ is
\[
N=\frac{1}{f(u)\sqrt{f(u)^2-\mid D u\mid^2}}\,\left(\,
f(u)^2\partial_t + Du \,\right).
\]
Then the hyperbolic angle $\varphi$, at any point of $M$, between
the unit timelike vectors $N$ and $\partial_t$, is given by
\begin{equation}\label{CBA}
\cosh \varphi= \frac{f(u)}{\sqrt{f(u)^2-\mid D u\mid^2}}
\end{equation}
and the corresponding mean curvature function is
\[
H(u)= \mathrm{div}\,\left(\frac{Du}{nf(u)\sqrt{f(u)^2-\mid
Du\mid^2}}\right) +\frac{f'(u)}{n\sqrt{f(u)^2 -\mid
Du\mid^2}}\left(n\,+\,\frac{\mid Du\mid^2}{f(u)^2}\right).
\]

In this section, our aim is to derive non-parametric uniqueness
results from the parametric ones provided in Section \ref{s4}. To do
that, we need the induced metric $g_u$ to be complete. Observe that,
in general, the induced metric on a closed spacelike hypersurface in
a complete Lorentzian manifold could be non-complete (see, for
instance, \cite{AM}). In our setting, we can derive the completeness
of $\Sigma_u$ as follows \cite[Lema 17]{ARR}

\begin{lema}\label{complete} Let $\overline{M}=I\times_f F$ be a GRW spacetime whose fiber
is a (non-compact) complete Riemannian manifold. Consider a function
$u\in C^{\infty}(F)$, with ${\rm Im}(u)\subseteq I$, such that the
entire graph $\Sigma_u=\{(u(p),p) \, : \, p\in F\}\subset
\overline{M}$ endowed with the metric $g_u=-du^2+f(u)^2g_{_F}$ is
spacelike. If the hyperbolic angle of $\Sigma_u$ is bounded and
$\inf f(u)>0$, then the graph $(\Sigma _u,g_{_{\Sigma_u}})$ is
complete, or equivalently the Riemannian surface $(F,g_u)$ is
complete.
\end{lema}

As a consequence of Theorem \ref{t2}, we have

\begin{teor} \label{CBt2}
Let $(F,g)$ be a simply connected parabolic Riemannian $n$-manifold,
$I\subseteq \mathbb{R}$ an open interval in $\mathbb{R}$ and
$f:I\longrightarrow \mathbb{R}^+$ a positive continuous function
satisfying that $Ric^F-(n-1)\, f^2 (\log f)''> 0$. Then the only
bounded entire solutions $u\in C^{\infty}(F)$, with ${\rm
Im}(u)\subseteq I$, to the uniformly elliptic
 non-linear differential equation
\[
H(u)=cte
\]
\begin{equation}\label{const}
\mid Du\mid<\lambda f(u), \quad 0<\lambda<1
\end{equation}
are the constant functions $u=u_0$ with $H=\frac{f'(u_0)}{f(u_0)}$.
\end{teor}
\begin{proof}
First observe that, from (\ref{CBA}), the constraint condition
(\ref{const}) can be written as
\begin{equation}\label{const2}
\cosh \varphi < \,\frac{1}{\sqrt{1-\lambda^2}}.
\end{equation}
Hence, (\ref{const}) holds if and only if $\Sigma_u$ has bounded
hyperbolic angle. Moreover, (\ref{const}) also implies that the
metric $g_u$ is spacelike, and furthermore it is complete from Lemma
\ref{complete}. Finally, the thesis follows from Theorem \ref{t2}.
\hfill{$\Box$}
\end{proof}

\begin{rem}
{\rm Note that the restriction (\ref{const2}) makes  $H(u)$ into a
uniformly elliptic operator.}
\end{rem}

For the particular case when $H(u)=0$, as a consequence of
Corollaries \ref{c1} and \ref{c2} we can state

\begin{coro}
Let $(F,g)$ be a simply connected parabolic Riemannian $n$-manifold,
$I\subseteq \mathbb{R}$ an open interval in $\mathbb{R}$ and
$f:I\longrightarrow \mathbb{R}^+$ a positive continuous function
satisfying that $Ric^F-(n-1)f^2 (\log f)''\geq 0$. Then the only
bounded entire solutions $u\in C^{\infty}(F)$, with ${\rm
Im}(u)\subseteq I$, to the uniformly elliptic
 non-linear differential equation
\[
H(u)=0
\]
\[
\mid Du\mid<\lambda f(u), \quad 0<\lambda<1
\]
are the totally geodesic (spacelike) graphs.

Furthermore, if $Ric^F-(n-1)f^2(\log f)''> 0$ then the only bounded
entire solutions are the constant functions $u=u_0$ with
$f'(u_0)=0$.
\end{coro}

From Theorem \ref{t3} we immediately obtain

\begin{teor} \label{CBt3}
Let $(F,g)$ be a simply connected parabolic Riemannian $n$-manifold,
$I\subseteq \mathbb{R}$ an open interval in $\mathbb{R}$ and
$f:I\longrightarrow \mathbb{R}^+$ a non locally constant positive
continuous function satisfying that $Ric^F-(n-1)f^2(\log f)''\geq
0$. Then the only bounded entire solutions $u\in C^{\infty}(F)$,
with ${\rm Im}(u)\subseteq I$, to the uniformly elliptic
 non-linear differential inequality
\[
H(u)^2=cte\geq \frac{f'(u)^2}{f(u)^2}
\]
\[
\mid Du\mid<\lambda f(u), \quad 0<\lambda<1
\]
are the constant functions $u=u_0$ with $H=\frac{f'(u_0)}{f(u_0)}$.
\end{teor}

Analogously, from Theorem \ref{tmon} we get

\begin{teor}\label{CBtmon}
Let $(F,g)$ be a simply connected parabolic Riemannian $n$-manifold,
$I\subseteq \mathbb{R}$ an open interval in $\mathbb{R}$ and
$f:I\longrightarrow \mathbb{R}^+$ a non-decreasing (resp.
non-increasing) positive continuous function satisfying that
$Ric^F-(n-1)f^2(\log f)''\geq 0$. Then the only bounded entire
solutions $u\in C^{\infty}(F)$, with ${\rm Im}(u)\subseteq I$, to
the to the uniformly elliptic
 non-linear differential equation
\[
H(u)=cte \leq 0 \qquad ({\rm resp.}\, H(u)=cte \geq 0)
\]
\[
\mid Du\mid<\lambda f(u), \quad 0<\lambda<1
\]
are the totally geodesic (spacelike) graphs.

Furthermore, if $Ric^F-(n-1)f^2(\log f)''> 0$ then the only bounded
entire solutions are the constant functions $u=u_0$ with
$f'(u_0)=0$.
\end{teor}

As a consequence of Theorem \ref{new}, we obtain (compare with
\cite[Th. 7.1]{Ca-Ro-Ru2}),

\begin{teor} Let $(F,g)$ be a simply connected parabolic Riemannian
$n$-manifold whose sectional curvature is bounded from below,
$I\subseteq \mathbb{R}$ an open interval in $\mathbb{R}$ and
$f:I\longrightarrow \mathbb{R}^+$ a positive smooth function
satisfying that $(\log f)''\leq 0$. Then the only bounded entire
solutions $u\in C^{\infty}(F)$, with ${\rm Im}(u)\subseteq I$, to
the uniformly elliptic
 non-linear differential equation
\[
H(u)=cte
\]
\[
\mid Du\mid<\lambda f(u), \quad 0<\lambda<1
\]
are the constant functions $u=u_0$ with $H=\frac{f'(u_0)}{f(u_0)}$.
\end{teor}

Finally, from Corollary \ref{corE} we can state

\begin{coro} Let $(F,g)$ be a simply connected parabolic Riemannian
$n$-manifold whose Ricci curvature is non-positive and whose
sectional curvature is bounded from below, $I\subseteq \mathbb{R}$
an open interval in $\mathbb{R}$ and $f:I\longrightarrow
\mathbb{R}^+$ one of the functions in cases 2-6 of Table
\ref{table1}. Then the only bounded entire solutions $u\in
C^{\infty}(F)$, with ${\rm Im}(u)\subseteq I$, to the uniformly
elliptic  non-linear differential equation
\[
H(u)=cte
\]
\[
\mid Du\mid<\lambda f(u), \quad 0<\lambda<1
\]
are the constant functions $u=u_0$ with $H=\frac{f'(u_0)}{f(u_0)}$.
\end{coro}

\section{Additional comments}\label{s5}

As is known, in an exact solution to the Einstein's field equation
the NCC follows from the weak energy condition, even if there is a
cosmological constant.

Conversely, consider a GRW spacetime $\overline{M}$ obeying the NCC
and $Z$ a timelike vector field on $\overline{M}$. Then from
(\ref{Ricb}) and (\ref{Sb}) we can compute the Einstein's tensor
$G=\overline{{\rm Ric}}-\frac{1}{2}\overline{S}\overline{g}$
evaluated at $Z$, so obtaining
\[
G(Z,Z)={\rm Ric}^ F (Z^ F,Z^ F)-(n-1)f^ 2(\log f)'' g_{_F} (Z^ F,Z^
F)-\frac{S^ F}{2f^ 2}\overline{g}(Z,Z)-\frac{n(n-1)}{2}\frac{f'^
2}{f^ 2}\overline{g}(Z,Z).
\]
Hence, $G(Z,Z)\geq 0$ when the scalar curvature of the fiber
satisfies $S^F+n(n-1)f'^2\geq 0$ or equivalently $S^F\geq
-n(n-1)\inf_I f'^2$. In particular, it holds when $S^F$ is
non-negative. Therefore, under this assumption on the scalar
curvature of the fiber a GRW spacetime obeying the NCC satisfies the
weak energy condition. Of course, the weak energy condition  will
also be satisfied if the Einstein's tensor includes the additional
term with non-negative cosmological constant.

Recall that the weak energy condition is a natural physical
assumption for normal matter. Thus, taking all of this into account,
we conclude that GRW spacetimes obeying the NCC and whose fiber has
non-negative scalar curvature can be suitable models for realistic
universes.

On the other hand, in a GRW spacetime there is a privileged family
of observer, that is the observers in the unitary timelike vector
field $\partial_t$, which moreover are proper time synchronizable.

For each $p\in F$ the curve $\gamma_{_p}(t)=(t,p)$ is the
\emph{worldline} or \emph{galaxy} of the corresponding observer in
$\partial_t$. Taking $t$ as a constant, we get the hypersurface
\[
M(t)=\{(t,p): p\in F\},
\]
which represents the physical space of the observer at the instant
$t$. Then, the distance between two galaxies $\gamma_{_p}$ and
$\gamma_{_q}$ in $M(t)$ is $f(t)d(p,q)$, where $d$ is the Riemannian
distance in the fiber $F$. In particular, when $f$ has positive
(resp. negative) derivative, the spaces $M(t)$ are expanding (resp.
contracting). Furthermore, if $f'>0$ and  $f''>0$ (resp. $f''<0$)
the GRW spacetime describes universes in accelerated (resp.
decelerated) expansion.

Recall that in a GRW spacetime the Timelike Energy Condition (TCC),
which is stronger than the NCC, implies that $f''\leq 0$. Therefore
GRW spacetimes obeying the TCC are not suitable models for
accelerated expanding universes. On the contrary, certain GRW
spacetimes obeying the NCC can be appropriate models for describing
such universes.

\section*{Acknowledgments}
The first author is partially supported by the Spanish MICINN Grant with FEDER funds MTM2013-
43970-P and by the Junta de Comunidades de Castilla-La Mancha Grant PEII-2014-001-A. The
second and third authors are partially supported by the Spanish MICINN Grant with FEDER funds MTM2013-
47828-C2-1-P.

\end{document}